\documentclass[12pt,reqno]{amsart}
\usepackage{geometry}                
\usepackage{amsmath,amssymb}
\usepackage{graphicx}
\usepackage{amssymb,latexsym, amsmath}
\usepackage{amsfonts,amsthm}
\usepackage{amscd}
\usepackage{mathrsfs}
\usepackage{hyperref}
\usepackage{eucal}
\usepackage{epstopdf}
\usepackage{amsgen}
\usepackage{xspace}
\usepackage{verbatim}
\usepackage{stmaryrd}
\usepackage{enumitem}
\newlist{steps}{enumerate}{1}
\setlist[steps, 1]{label = Step \arabic*:}

\newcommand{\gd}{\Delta}

\newcommand{\inpt}[1]{\langle #1 \rangle}

\newcommand{\gw}{\Omega}
\newcommand{\ap}{\alpha}

\newcommand{\gb}{\beta}
\newcommand{\G}{\Gamma}

\newcommand{\ms}{\mathscr}

\newcommand{\nb}{\nabla}

\newcommand{\pdr}{\partial}

\newcommand{\beq}{\begin{equation}}
\newcommand{\eeq}{\end{equation}}
\newcommand{\bea}{\begin{align}}
\newcommand{\eea}{\end{align}}
\newcommand{\bthm}{\begin{theorem}}
\newcommand{\ethm}{\end{theorem}}
\newcommand{\bpr}{\begin{proof}}
\newcommand{\epr}{\end{proof}}
\newcommand{\bcl}{\begin{corollary}}
\newcommand{\ecl}{\end{corollary}}
\newcommand{\bpn}{\begin{proposition}}
\newcommand{\epn}{\end{proposition}}
\newcommand{\bre}{\begin{remark}}
\newcommand{\ere}{\end{remark}}
\newcommand{\bdf}{\begin{definition}}
\newcommand{\edf}{\end{definition}}
\newcommand{\bss}{\begin{align*}}
\newcommand{\ess}{\end{align*}}

\newcommand{\bl}{\label}

\newcommand{\mR}{\mathbb{R}}
\newcommand{\mH}{\mathbb{H}}

\newcommand{\mE}{\mathbb{E}}

\newtheorem{theorem}{Theorem}[section]
\newtheorem{corollary}[theorem]{Corollary}

\newtheorem{proposition}[theorem]{Proposition}

\theoremstyle{definition}
\newtheorem{definition}[theorem]{Definition}
\theoremstyle{remark}
\newtheorem{remark}{Remark}

\numberwithin{equation}{section}

\begin{document}

\title[Synchronization of Neuron Network]{Synchronization of Boundary Coupled Hindmarsh-Rose Neuron Network}

\author[C. Phan]{Chi Phan}
\address{Department of Mathematics and Statistics, University of South Florida, Tampa, FL 33620, USA}
\email{chi@mail.usf.edu}
\thanks{}

\author[Y. You]{Yuncheng You}
\address{Department of Mathematics and Statistics, University of South Florida, Tampa, FL 33620, USA}
\email{you@mail.usf.edu}
\thanks{}

\subjclass[2010]{35B40, 35B45, 35K58, 35M33, 35Q92, 92B25, 92C20}

\date{December 10, 2019}


\keywords{Neuron network, synchronization with boundary coupling, Hindmarsh-Rose equations, absorbing semiflow.}

\begin{abstract} 
In this work, we present a new mathematical model of a boundary coupled neuron network described by the partly diffusive Hindmarsh-Rose equations. We prove the global absorbing property of the solution semiflow and then the main result on the asymptotic synchronization of this neuron network at a uniform exponential rate provided that the boundary coupling strength and the stimulating signal exceed a quantified threshold in terms of the parameters.
\end{abstract}

\maketitle

Synchronization of biological neurons is one of the central topics in neuroscience. Here we shall present a new mathematical model of multiple Hindmarsh-Rose neurons with boundary coupling, which yields the asymptotic synchronization at a uniform exponential rate independent of initial states.

The original Hindmarsh-Rose equations for single neuron firing-bursting were proposed in \cite{HR}. That neuron model of ordinary differential equations has been studied through numerical simulations and bifurcation analysis, cf. \cite{DJ, IG, EI, ET, MFL, SPH, Tr, WS, Su} and the references therein. Hodgkin-Huxley equations \cite{HH} and FitzHugh-Nagumo equations \cite{FH} also provided mathematical models for single neuron dynamics.

Neuronal signals are short electrical pulses called spikes or action potential. Neurons often exhibit bursts of alternating phases of rapid firing spikes and then quiescence. Bursting patterns modulates the brain functionalities and are experimentally observed in many bio-systems, cf. \cite{BRS, CK,CS, HR}. Synchronization of bursting for multiple neurons plays a key role in an effective execution of the command from central nerve system.

In recent work \cite{Phan, PY, PYS, PYY}, the authors studied the single neuron model of diffusive Hindmarsh-Rose equations and proved the existence of global attractor and pullback random attractor for the solution semiflow and the solution cocycle. 

Mathematical testimony for synchronization of coupled neurons by a hybrid model of partly diffusive partial-ordinary differential equations is an open problem. The new model of boundary coupled neuron network presented in this paper reflects the structural feature of neuron cells, especially the short-branch dendrites receiving incoming signals and the long-branch axon propagating outreaching signals as well as that neurons are immersed in aqueous biochemical solutions with charged ions. 

\section{\textbf{New Model of Boundary Coupled Neuron Network}}

In this paper, we present a new model of boundary coupled neuron network in terms of the following system of the partly diffusive Hindmarsh-Rose equations,
\beq \bl{cHR}
\begin{split}
	\frac{\pdr u}{\pdr t} & = d \gd u +  au^2 - bu^3 + v - w + J,  \\
	\frac{\pdr v}{\pdr t} & =  \alpha - v - \beta u^2,    \\
	\frac{\pdr w}{\pdr t} & = q (u - c) - rw,   \\
	\frac{\pdr u_i}{\pdr t} & = d \gd u_i + au_i^2 - bu_i^3 + v_i - w_i + J,  \quad 1 \leq i \leq m, \\
	\frac{\pdr v_i}{\pdr t} & =  \alpha - v_i - \beta u_i^2,  \quad 1 \leq i \leq m,  \\
	\frac{\pdr w_i}{\pdr t} & = q (u_i - c) - rw_i, \quad 1 \leq i \leq m,
\end{split}
\eeq
for $t > 0,\; x \in \gw \subset \mathbb{R}^{n}$ ($n \leq 3$), where $\gw$ is a bounded domain and its boundary 
$$
	\partial \gw = \G = \bigcap_{i = 0}^m \G_i
$$
is locally Lipschitz continuous, where the boundary pieces $\G_i, i = 0, 1, \cdots, m,$ are measurable and mutually non-overlapping. Here $(u_i, v_i, w_i), \,i = 1, \cdots, m,$ are the state variables for the \emph{neighbor neurons} denoted by $\mathcal{N}_i, i = 1, \cdots, m$, coupled with the \emph{central neuron} denoted by $\mathcal{N}_c$ whose state variables are $(u, v, w)$. 

In this system \eqref{cHR}, the variables $u(t, x)$ and $u_i(t,x)$ refer to the membrane electrical potential of a neuron cell, the variables $v(t, x)$ and $v_i(t, x)$ called the spiking variables represent the transport rate of the ions of sodium and potassium through the fast ion channels, and the variables $w(t, x)$ and $w_i(t, x)$ called the bursting variables represent the transport rate across the neuron cell membrane through slow channels of calcium and other ions.

The coupling boundary conditions affiliated with the system \eqref{cHR} are given by
\beq \label{nbc}
	\begin{split}
	&\frac{\pdr u}{\pdr \nu} (t, x) = 0, \;\; \text{for} \; x \in \G_0,  \;\;  \frac{\pdr u}{\pdr \nu} (t, x) + pu = pu_i, \;\, \text{for} \; x \in \G_i, \; 1 \leq i \leq m;  \\
	&\frac{\pdr u_i}{\pdr \nu} (t, x) = 0, \;\, \text{for} \; x \in \G \backslash \G_i,  \;\, \frac{\pdr u_i}{\pdr \nu} (t, x) + p u_i = pu, \;\, \text{for}  \; x \in \G_i, \; 1 \leq i \leq m.
	\end{split}
\eeq
where $\pdr/\pdr \nu$ stands for the normal outward derivative, $p > 0$ is the coupling strength constant and the switch functions of the neuron $\mathcal{N}_i, 1 \leq i \leq m$, are 
$$
	\xi_i (x) = \begin{cases}
		1,  \vspace{5pt} & \text{if  \, $x \in \G_i$}; \\[3pt]
		0,  \vspace{5pt} & \text{if  \, $x \in \G \backslash \G_i$}.
	\end{cases}
$$ 
The initial conditions to be specified are denoted by 
\beq \bl{inc}
	\begin{split}
	&u(0, x) = u^0 (x), \;\quad v(0, x) = v^0 (x), \quad \, w(0, x) = w^0 (x), \quad  \; x \in \gw, \\[3pt]
	&u_i(0, x) = u_i^0 (x), \quad v_i(0, x) = v_i^0 (x), \quad w_i (0, x) = w_i^0 (x), \quad x \in \gw,
	\end{split}
\eeq
for $1 \leq i \leq m$.

All the parameters in this system \eqref{cHR} including the input electrical current $J$ are positive constants except a reference value of the membrane potential of neuron cells $c = u_R \in \mR$.

In this study of the neuron network \eqref{cHR}-\eqref{inc}, we shall work with the following Hilbert spaces for the subsystem of three equations for each involved single neuron:
$$
	H = L^2 (\gw, \mR^3), \quad  \text{and} \quad E = H^1 (\gw) \times L^2 (\gw, \mR^2).
$$
Also define the product spaces 
$$
	\mathbb{H} =  [L^2 (\gw, \mR^3)]^{1+m} \quad \text{and} \quad \mathbb{E} = [H^1 (\gw) \times L^2 (\gw, \mR^2)]^{1+m}
$$
for the entire system \eqref{cHR}-\eqref{inc}. The norm and inner-product of the Hilbert space $\mathbb{H}, \, H$ or $L^2(\gw)$ will be denoted by $\| \, \cdot \, \|$ and $\inpt{\,\cdot , \cdot\,}$, respectively. The norm of $\mathbb{E}$ or $E$ will be denoted by $\| \, \cdot \, \|_E$. We use $| \, \cdot \, |$ to denote the vector norm or the measure of set in $\mR^n$.

The initial-boundary value problem \eqref{cHR}-\eqref{inc} can be formulated as an initial value problem of the evolutionary equation:
\begin{equation} \label{pb}
\begin{split}
	\frac{\partial}{\partial t} \begin{pmatrix}
	g\\
	g_i
	\end{pmatrix}
	=  \begin{pmatrix}
	A & 0 \\
	0 & A_i
	\end{pmatrix} &
	 \begin{pmatrix}
	g\\
	g_i
	\end{pmatrix}
	+  \begin{pmatrix}
	 f (g) \\
	 f (g_i) 
	\end{pmatrix}, \;\;  1 \leq i \leq m, \;\; t > 0, \\[4pt]
	&g(0) = g^0 \in H, \quad g_i (0) = g_i^0 \in H.
\end{split}
\end{equation}
Here $g(t) = \text{col}\, (u(t, \cdot), v(t, \cdot ), w(t, \cdot))$ and $g_i(t) = (u_i(t, \cdot), v_i (t, \cdot ), w_i(t, \cdot))$. The initial data functions are $g^0 = \text{col}\,(u^0, v^0, w^0)$ and $g_i^0 = \text{col}\, (u_i^0, v_i^0, w_2^0),$ for $1 \leq i \leq m$. The nonpositive, self-adjoint and diagonal operator $\mathcal{A} = \text{diag}\, (A, A_1, \cdots, A_m)$ is defined by the block operators
\begin{equation} \label{opA}
	A = A_i =
\begin{bmatrix}
d \gd \quad & 0  \quad & 0 \\[3pt]
0 \quad & - I \quad  & 0 \\[3pt]
0 \quad & 0 \quad & - r I 
\end{bmatrix}, \quad 1 \leq i \leq m,
\end{equation}
with the domain
$$
	D(\mathcal{A}) = \{ \text{col}\, (h, h_1, \cdots, h_m) \in [H^2(\gw) \times L^2 (\gw, \mR^2)]^{1+m}: \text{\eqref{nbc} satisfied}\}.
$$ 
Due to the continuous Sobolev imbedding $H^{1}(\gw) \hookrightarrow L^6(\gw)$ for space dimension $n \leq 3$ and by the H\"{o}lder inequality, the nonlinear mapping 
\begin{equation} \label{opf}
\begin{pmatrix}
	f(g) \\[3pt]
	f(g_i)
\end{pmatrix} =
\begin{pmatrix}
au_1^2 - bu_1^3 + v_1 - w_1 + J \\[3pt]
\alpha - \beta u_1^2  \\[3pt]
q (u_1 - c) \\[3pt]
au_2^2 - bu_2^3 + v_2 - w_2 + J \\[3pt]
\alpha - \beta u_2^2  \\[3pt]
q (u_2 - c)
\end{pmatrix}
: E \times E \longrightarrow H \times H
\end{equation}
is a locally Lipschitz continuous mapping for $1 \leq i \leq m$. 

We shall consider the weak solution of this initial value problem \eqref{pb}, cf. \cite[Section XV.3]{CV} and the corresponding definition presented in \cite{PY, PYS}. The following proposition can be proved by the Galerkin approximation method.

\begin{proposition} \label{pps}
	For any given initial state $(g^0, g_1^0, \cdots , g_m^0) \in \mathbb{H}$, there exists a unique local weak solution $(g(t, g^0), g_1 (t, g_1^0), \cdots, g_m (t, g_m^0)), \, t \in [0, \tau]$, for some $\tau > 0$, of the initial value problem \eqref{pb} formulated from the problem \eqref{cHR}-\eqref{inc}. The weak solution continuously depends on the initial data and satisfies 
	\begin{equation} \label{soln}
	(g, g_1, \cdots, g_m)  \in C([0, \tau]; \,\mH) \cap C^1 ((0, \tau); \,\mH) \cap L^2 ([0, \tau]; \,\mE).
	\end{equation}
If the initial state is in $\mE$, then the solution is a strong solution with the regularity
	\begin{equation} \bl{ss}
	(g, g_1, \cdots, g_m) \in C([0, \tau]; \,\mE) \cap C^1 ((0, \tau); \,\mE) \cap L^2 ([0, \tau]; \,D(A) \times D(A_i)^m).
	\end{equation}
\end{proposition}

The basics of infinite dimensional dynamical systems or called semiflow generated by parabolic partial differential equations are referred to \cite{CV, SY, Tm}.

\begin{definition} \label{Dabsb}
	Let $\{S(t)\}_{t \geq 0}$ be a semiflow on a Banach space $\ms{X}$. A bounded set $B^*$ of $\ms{X}$ is called an absorbing set of this semiflow, if for any given bounded set $B \subset \ms{X}$ there exists a finite time $T_B \geq 0$ depending on $B$, such that $S(t)B \subset B^*$ permanently for all $t \geq T_B$. 
\end{definition}

\section{\textbf{Global Existence of Solutions and Absorbing Semiflow}}

First we prove the global existence of weak solutions in time for the initial value problem \eqref{pb} of the boundary coupled partly diffusive Hindmarsh-Rose equations.

\begin{theorem} \label{Tm}
	For any given initial state $(g^0, g_1^0, \cdots , g_m^0) \in \mathbb{H}$, there exists a unique global weak solution in time, $(g(t), g_1 (t), \cdots, g_m(t)), \, t \in [0, \infty)$, of the initial value problem \eqref{pb} formulated from the original initial-boundary value problem \eqref{cHR}-\eqref{inc}. 
\end{theorem}

\begin{proof}
	Sum up the $L^2$ inner-products of the $u$-equation with $C_1 u(t)$ and the $u_i$-equation with $C_1 u_i(t)$ for $1 \leq i \leq m$, the constant $C_1 > 0$ to be chosen, we get
	\begin{equation*}
	\begin{split}
	&\frac{C_1}{2} \frac{d}{dt} \left(\|u \|^2 + \sum_{i = 1}^m \|u_i\|^2 \right) + C_1 d \left(\| \nabla u \|^2 + \sum_{i=1}^m \|\nb u_i\|^2 \right) \\
	= &\, C_1 \int_\gw \left[(au^3 -bu^4  + u v - u w +J u) + \sum_{i=1}^m (au_i^3 -bu_i^4  + u_i v_i - u_i w_i +Ju_i)\right] dx \\
	+ &\,d C_1\, \sum_{i=1}^m \int_{\G_i} \left(p(u_i - u)u + p(u - u_i) u_i \right) \, dx \\
	= &\, \int_\gw C_1 (au^3 -bu^4  + u v - u w +J u)\, dx \\
	+ &\, \sum_{i=1}^m \int_\gw (C_1 (au_i^3 -bu_i^4  + u_i v_i - u_i w_i +Ju_i) \, dx - d C_1p \, \sum_{i=1}^m \int_{\G_i} ( u - u_i)^2\, dx \\
	\leq &\, C_1 \int_\gw \left[(au^3 -bu^4  + u v - u w +J u) + \sum_{i=1}^m (au_i^3 -bu_i^4  + u_i v_i - u_i w_i +Ju_i)\right] dx,
	\end{split}
	\end{equation*}
by the coupling boundary condition \eqref{nbc}. Then sum up the $L^2$ inner-products of the $v$-equation with $v (t)$ and the $v_i$-equation with $v_i(t)$ for $1 \leq i \leq m$, we obtain
	\begin{equation*}
	\begin{split}
	&\frac{1}{2} \frac{d}{dt} (\|v \|^2 + \sum_{i=1}^m \| v_i\|^2) = \int_\gw \left[(\ap v - \gb u^2 v - v^2 + \sum_{i=1}^m (\ap v_i - \gb u_i^2 v_i - v_i^2) \right]dx \\
	\leq &\int_\gw \left[\ap v +\frac{1}{2} (\gb^2 u^4 + v^2) - v^2 +  \sum_{i=1}^m (\ap v_i +\frac{1}{2} (\gb^2 u_i^4 + v_i^2) - v_i^2)\right] dx \\
	\leq &\int_\gw \left[(1 + m)\ap^2 +\frac{1}{2} \gb^2 (u^4 +  \sum_{i=1}^m u_i^4) - \frac{3}{8} (v^2 +  \sum_{i=1}^m v_i^2) \right] dx,
	\end{split}
	\end{equation*}
and similarly for the $w$-equation and $w_i$-equation, $1 \leq i \leq m$, we have
	\begin{equation*} 
	\begin{split}
	&\frac{1}{2} \frac{d}{dt} (\|w \|^2 + \sum_{i=1}^m \| w_i \|^2) = \int_\gw \left[(q (u - c)w - rw^2) +  \sum_{i=1}^m (q (u_i - c)w_i - rw_i^2) \right] dx  \\
	\leq & \int_\gw \left[\frac{q^2}{2r} (u - c)^2 + \frac{1}{2} r w^2 - r w^2 +  \sum_{i=1}^m \left(\frac{q^2}{2r} (u_i - c)^2 + \frac{1}{2} r w_i^2 - r w_i^2\right)\right] dx \\
	\leq & \int_\gw \left[\frac{q^2}{r} \left(u^2 +  \sum_{i=1}^m u_i^2 + (1 + m)c^2\right) - \frac{r}{2} \left(w^2 +  \sum_{i=1}^m w_i^2\right)\right] dx.
	\end{split}
	\end{equation*}

To treat the nonlinear integral terms on the right-hand side of the first inequality above, we choose the positive constant to be $C_1 = \frac{1}{b} (\gb^2 + 4)$. Then 
	\beq \bl{C1u}
		\begin{split}
		&\int_\gw (- C_1 b u^4)\, dx + \int_\gw (\gb^2 u^4)\, dx \leq \int_\gw (-4 u^4)\, dx, \\
		&\int_\gw (- C_1 b u_i^4)\, dx + \int_\gw (\gb^2 u_i^4)\, dx \leq \int_\gw (-4 u_i^4)\, dx, \quad i= 1, \cdots, m.
		\end{split}
	\eeq
Using the Young's inequality in an appropriate way, we deduce that 
	\beq \bl{C3u}
		\begin{split}
		&\int_\gw C_1 au_i^3\, dx \leq \frac{3}{4} \int_\gw u^4\, dx + \frac{1}{4}\int_\gw (C_1 a)^4 \, dx \leq \int_\gw u^4\, dx + (C_1 a)^4 |\gw|, \\
		&\int_\gw C_1 au_i^3\, dx \leq \frac{3}{4} \int_\gw u_i^4\, dx + \frac{1}{4}\int_\gw (C_1 a)^4 \, dx \leq \int_\gw u_i^4\, dx + (C_1 a)^4 |\gw|, 
		\end{split}
	\eeq
for $i = 1, \cdots, m$. Moreover, we have
	\beq \bl{uvw}
	\begin{split}
	&C_1 \int_\gw \left( (uv -uw + ju) + \sum_{i=1}^m (u_i v_i - u_i w_i + Ju_i)\right) dx \\[2pt]
	\leq &\, \int_\gw \left(2(C_1 u)^2 + \frac{1}{8} v^2 + \frac{(C_1 u)^2}{r} + \frac{1}{4} r w^2 + C_1 u^2 + C_1J^2 \right) dx \\
	+ &\, \int_\gw \sum_{i=1}^m \left(2(C_1 u_i)^2 + \frac{1}{8} v_i^2 + \frac{(C_1 u_i)^2}{r} + \frac{1}{4} r w_i^2 + C_1 u_i^2 + C_1J^2 \right) dx 	\end{split}
	\eeq
	where on the right-hand side of the inequality \eqref{uvw} we can further treat the terms involving $u^2$ and $u_i^2$ as follows,
	\beq \bl{ur}
	\begin{split}
	&\int_\gw \left(2(C_1 u)^2 + \frac{(C_1 u)^2}{r} + C_1 u^2 + \sum_{i=1}^m \left[2(C_1 u_i)^2 + \frac{(C_1 u_i)^2}{r} + C_1 u_i^2\right] \right) dx \\
	 \leq &\, \int_\gw \left(u^4 + \sum_{i=1}^m u_i^4 \right) dx + (1 + m)\left[C_1^2 \left(2 +\frac{1}{r}\right) + C_1\right]^2 |\gw |.
	\end{split}
	\eeq
	Besides we have
	\beq \bl{uq}
	\int_\gw \frac{1}{r} q^2 \left(u^2 + \sum_{i=1}^m u_i^2\right) dx \leq \int_\gw \left(u^4 + \sum_{i=1}^m u_i^4\right) dx + \frac{q^4}{r^2}(1 + m) |\gw|.
	\eeq
Substitute the estimates \eqref{C1u} -\eqref{uq} into the first three differential inequalities in this proof and then sum them up to obtain
	\beq \label{g2}
	\begin{split}
		&\frac{1}{2} \frac{d}{dt} \left[C_1 \left(\|u\|^2 +  \sum_{i = 1}^m \|u_i\|^2\right) +  \left(\|v\|^2 +  \sum_{i=1}^m \| v_i\|^2\right) + \left(\|w\|^2 +  \sum_{i=1}^m \| w_i \|^2\right) \right] \\
		&\; + C_1 d\, \left(\|\nb u \|^2 + \sum_{i=1}^m \|\nb u_i \|^2\right)  \\[2pt]
		\leq &\, C_1 \int_\gw \left[(au^3 -bu^4  + u v - u w +J u) + \sum_{i=1}^m (au_i^3 -bu_i^4  + u_i v_i - u_i w_i +Ju_i)\right] dx \\
		&+ \int_\gw \left[(1 + m)\ap^2 +\frac{1}{2} \gb^2 (u^4 +  \sum_{i=1}^m u_i^4) - \frac{3}{8} (v^2 +  \sum_{i=1}^m v_i^2) \right] dx \\
		&+ \int_\gw \left[\frac{q^2}{r} \left(u^2 +  \sum_{i=1}^m u_i^2 + (1 + m)c^2\right) - \frac{r}{2} \left(w^2 +  \sum_{i=1}^m w_i^2\right)\right] dx \\
		\leq & \int_\gw (3 - 4)\left(u^4 + \sum_{i=1}^m u_i^4 \right) dx + \int_\gw \left(\frac{1}{8} - \frac{3}{8}\right) \left(v^2 + \sum_{i=1}^m v_i^2\right) dx \\
		&+ \int_\gw \left(\frac{1}{4} - \frac{1}{2} \right) r \left(w^2 + \sum_{i=1}^m w_i^2 \right) dx \\
		&+  \, (1 + m)|\gw | \left( (C_1 a)^4 + C_1 J^2  + \left[C_1^2 \left(2 +\frac{1}{r}\right) + C_1\right]^2 + 2\ap^2 + \frac{q^2 c^2}{r} + \frac{q^4}{r^2} \right) \\
		= &\, - \int_\gw \left(\left[u^4 + \sum_{i=1}^m u_i^4 \right] + \frac{1}{4} \left[v^2 + \sum_{i=1}^m v_i^2\right] + \frac{r}{4} \left[w^2 + \sum_{i=1}^m w_i^2 \right] \right) dx + C_2 (1 + m)|\gw |, 	
	\end{split}
	\eeq
	where $C_2 = 2(C_1 a)^4 + 2C_1 J^2  + 2\left[C_1^2 \left(2 +\frac{1}{r}\right) + C_1\right]^2 + 4\ap^2 + \frac{2q^2 c^2}{r} + \frac{2q^4}{r^2}$ is a constant. 
From \eqref{g2} it follows that
	\beq \label{E1}
	\begin{split}
		&\frac{d}{dt} \left[C_1 \left(\|u\|^2 +  \sum_{i = 1}^m \|u_i\|^2\right) +  \left(\|v\|^2 +  \sum_{i=1}^m \| v_i\|^2\right) + \left(\|w\|^2 +  \sum_{i=1}^m \| w_i \|^2\right) \right] \\
		+ \,2& \int_\gw \left(\left[u^4 + \sum_{i=1}^m u_i^4 \right] + \frac{1}{4} \left[v^2 + \sum_{i=1}^m v_i^2\right] + \frac{r}{4} \left[w^2 + \sum_{i=1}^m w_i^2 \right] \right) dx \leq C_2 (1 + m)|\gw|, 
	\end{split}
	\eeq
for $t \in I_{max} = [0, T_{max})$, the maximal time interval of solution existence. Note that in the first part of the integral term of \eqref{E1} we have
	$$ 
	  \frac{1}{4} \left(C_1 u^2 - \frac{C_1^2}{16}\right) \leq u^4 \quad \text{and} \quad \frac{1}{4} \left(C_1 u_i^2 - \frac{C_1^2}{16}\right) \leq u_i^4, \quad 1 \leq i \leq m.
	$$
Then \eqref{E1} yields the following differential inequality
	\beq \bl{E2}
	\begin{split}
	 &\frac{d}{dt} \left[C_1 \left(\|u\|^2 +  \sum_{i = 1}^m \|u_i\|^2\right) +  \left(\|v\|^2 +  \sum_{i=1}^m \| v_i\|^2\right) + \left(\|w\|^2 +  \sum_{i=1}^m \| w_i \|^2\right) \right] \\
		&+ \, r^* \left[C_1 \left(\|u\|^2 +  \sum_{i = 1}^m \|u_i\|^2\right) +  \left(\|v\|^2 +  \sum_{i=1}^m \| v_i\|^2\right) + \left(\|w\|^2 +  \sum_{i=1}^m \| w_i \|^2\right) \right]  \\
	&\leq \frac{d}{dt} \left[C_1 \left(\|u\|^2 +  \sum_{i = 1}^m \|u_i\|^2\right) +  \left(\|v\|^2 +  \sum_{i=1}^m \| v_i\|^2\right) + \left(\|w\|^2 +  \sum_{i=1}^m \| w_i \|^2\right) \right] \\
		&+ \, \frac{1}{2} \int_\gw \left(\left[u^2 + \sum_{i=1}^m u_i^2 \right] + \left[v^2 + \sum_{i=1}^m v_i^2\right] + r \left[w^2 + \sum_{i=1}^m w_i^2 \right] \right) dx \\
		&\leq \left(C_2 + \frac{C_1^2}{32}\right)(1 + m) |\gw |,
	\end{split}
	\eeq
where $r^* = \frac{1}{2} \min \{1, r\}$. Apply the Gronwall inequality to \eqref{E2}. Then we obtain the following bounding estimate of the weak solutions:
	\beq \label{dse}
	\begin{split}
		&\|g(t, g^0)\|^2 + \sum_{i=1}^m \|g_i (t, g_i^0)\|^2 \\
		\leq &\, \frac{\max \{C_1, 1\}}{\min \{C_1, 1\}}e^{- r^* t} \left(\|g^0\|^2 + \sum_{i=1}^m \|g_i^0\|^2\right) + \frac{M}{\min \{C_1, 1\}} |\gw | \\
		\leq &\, \frac{\max \{C_1, 1\}}{\min \{C_1, 1\}} \left(\|g^0\|^2 + \sum_{i=1}^m \|g_i^0\|^2\right) + \frac{M}{\min \{C_1, 1\}} |\gw |
	\end{split}
	\eeq 
for $t \in I_{max} = [0, T_{max}) = [0, \infty)$, where 
\beq \bl{M}
	M = \frac{1 + m}{r^*}\left(C_2 + \frac{C_1^2}{32}\right).
\eeq
The estimate \eqref{dse} shows that the weak solution $g(t, x)$ will never blow up at any finite time because it is bounded uniformly on the existence time interval. Therefore, for any initial data in $\mH$, the unique weak solution of the initial value problem \eqref{pb} of the boundary coupled neuron network \eqref{cHR}-\eqref{inc} exists in $\mH$ globally in time.
\end{proof}

The global existence and uniqueness of the weak solutions and their continuous dependence on the initial data enable us to define the solution semiflow $\{S(t): \mH \to \mH\}_{t \geq 0}$ of the boundary coupled Hindmarsh-Rose neuron network system \eqref{cHR}-\eqref{inc} on the space $\mH$ as follows,
\beq \bl{HRS}
	S(t): (g^0, g_1^0, \cdots, g_m^0) \longmapsto (g(t, g^0), g_1(t, g_1^0), \cdots, g_m(t, g_m^0)), \quad  t \geq 0.
\eeq
We call this semiflow $\{S(t)\}_{t \geq 0}$ the \emph{boundary coupling Hindmarsh-Rose semiflow}. 

\begin{theorem} \label{Hab}
	There exists an absorbing set for the boundary coupling Hindmarsh-Rose semiflow $\{S(t)\}_{t \geq 0}$ in the space $\mH$, which is the bounded ball 
\beq \label{abs}
	B^* = \{ h \in \mH: \| h \|^2 \leq Q\}
\eeq 
	where $Q = \frac{M |\gw |}{\min \{C_1, 1\}} + 1$.
\end{theorem}

\begin{proof}
This is the consequence of the uniform estimate \eqref{dse} in Theorem \ref{Tm} because
	\beq \label{lsp}
	\limsup_{t \to \infty} \, \left(\|g(t)\|^2 + \sum_{i=1}^m \|g_i(t)\|^2\right) < Q = \frac{M |\gw |}{\min \{C_1, 1\}} + 1
	\eeq
	for all weak solutions of \eqref{pb} with any initial data in $\mH$. Moreover, for any given bounded set $B = \{h \in \mH: \|h \|^2 \leq \rho \}$ in $\mH$, there exists a finite time 
	\beq \label{T0B}
	T_0 (B) = \frac{1}{r^*} \log^+ \left(\rho \, \frac{\max \{C_1, 1\}}{\min \{C_1, 1\}}\right)
	\eeq
	such that all the solution trajectories started from the set $B$ will permanently enter the bounded ball $B^*$ shown in \eqref{abs} for $t \geq T_0(B)$.
\end{proof}

\section{\textbf{Synchronization of the Boundary Coupled Neuron Neiwork}} 

Synchronization for ensemble of neurons and for complex neuron network or some artificial neural network is one of the central and significant topics in neuroscience and in the theory of artificial intelligence. 

We introduced a new concept of synchronization dynamics for a neuron network.
\begin{definition} \bl{DaD}
	For the dynamical system generated by a model differential equation such as \eqref{pb} of multiple neurons with whatever type of coupling, define the \emph{asynchronous degree} in a state space $\ms{X}$ to be
	$$
	deg_s (\ms{X})= \sum_{j} \sum_{k} \, \sup_{g_j^0, \, g_k^0 \in \ms{X}} \, \left\{\limsup_{t \to \infty} \, \|g_j (t) - g_k(t)\|_{\ms{X}}\right\},
	$$ 
where $g_j(t)$ and $g_k(t)$ are any two solutions of the model differential equation with the initial states $g_j^0$ and $g_k^0$, respectively. Then the coupled neuron network is said to be asymptotically synchronized in the space $\ms{X}$, if $deg_s (\ms{X}) = 0$.
\end{definition}

In this section, we shall prove the main result of this work on the asymptotic synchronization of the boundary coupled Hindmarsh-Rose neuron network described by \eqref{cHR}-\eqref{inc} in the space $H$. This result provides a quantitative threshold for the coupling strength and the stimulation signals to reach the asymptotic synchronization.

To address mathematically this synchronization problem of the neuron network specified in Section 1, denote by $U_i(t) = u(t) - u_i (t), V_i(t) = v(t) - v_i(t), W_i(t) = w(t) - w_i(t)$, for $i = 1, \cdots, m$. Then for any given initial states $g^0$ and $g_i^0, \cdots, g_m^0$ in the space $H$, the difference between the solutions associated with the neuron $\mathcal{N}_c$ and the neuron $\mathcal{N}_i$ is
	$$
	g (t, g^0) - g_i (t, g_i^0) = \text{col}\, (U_i(t), V_i(t), W_i(t)), \quad t \geq 0.
	$$
	
	By subtraction of the corresponding three pairs of equations of the $i$-th neuron from the central neuron in \eqref{cHR}, we obtain the differencing Hindmarsh-Rose equations as follows. For $i = 1, \cdots, m$,
\beq \bl{dHR}
	\begin{split}
		\frac{\pdr U_i}{\pdr t} & = d \gd U_i +  a(u + u_i)U_i - b(u^2 + u u_i + u_i^2)U_i + V_i - W_i,  \\
		\frac{\pdr V_i}{\pdr t} & =  - V_i - \beta (u +  u_i)U_i,    \\
		\frac{\pdr W_i}{\pdr t} & = q U_i - r W_i.
	\end{split}
\eeq
Here is the main result on the synchronization of the boundary coupled Hindmarsh-Rose neuron network.
\begin{theorem} \bl{ThM}
	If the threshold condition for stimulation signal strength of the boundary coupled Hindmarsh-Rose neuron network is satisfied that for any given initial conditions $g^0, g_i^0 \in H$, 
\beq \bl{SC}
	p\, \liminf_{t \to \infty} \,\int_{\G_i} U_i^2(t, x)\, dx > R \, |\gw|, \quad i = 1, \cdots, m,
\eeq
where
\beq \bl{R}
	R = \frac{1 + m}{r^* \min \{C_1, 1\}}\left[\frac{C_1^2}{32} + C_2\right] \left[\eta_2\, d \,|\gw | + \left[\frac{8\beta^2}{b} + \frac{a^2}{b} + \frac{b}{16\beta^2 r} \left[q - \frac{8\beta^2}{b}\right]^2\right] \right]
\eeq
with $C_1 = \frac{1}{b} (\beta^2 + 4)$, $\eta_2 > 0$ being the constant in Poincar\'{e} inequality \eqref{Pcr}, and
\beq  \bl{C2}
	C_2 = 2(C_1 a)^4 + 2C_1 J^2  + 2\left[C_1^2 \left(2 +\frac{1}{r}\right) + C_1\right]^2 + 4\ap^2 + \frac{2q^2 c^2}{r} + \frac{2q^4}{r^2},
\eeq
then the boundary coupled Hindmarsh-Rose neuron network generated by \eqref{pb} is asymptotically synchronized in the space $H$ at a uniform exponential rate.
\end{theorem}

\begin{proof}
	Step 1. Take the $L^2$ inner-products of the first equation in \eqref{dHR} with $KU_i(t)$, the second equation in \eqref{dHR} with $V_i(t)$, and the third equation in \eqref{dHR} with $W_i(t)$, where $K > 0$ to be chosen. Then sum them up and use Young's inequalities to get
\beq \bl{eG}
	\begin{split}
		&\frac{1}{2} \frac{d}{dt} (K\|U_i (t)\|^2 + \|V_i (t)\|^2 + \|W_i (t)\|^2) + d K \|\nb U_i (t)\|^2  + \|V_i (t)\|^2 + r\, \|W_i (t)\|^2 \\[3pt]
		= &\,\int_{\G} K\frac{\pdr U_i}{\pdr \nu}\, U_i \, dx +  \int_\gw K (a(u + u_i)U_i^2 - b(u^2 + u u_i + u_i^2) U_i^2 )\, dx \\[2pt]
		+ &\, \int_\gw \left(K U_i V_i -\beta (u +  u_i)U_i V_i + (q - K)U_i W_i \right) dx \\[2pt]
		\leq &\,\int_{\G} K\frac{\pdr U_i}{\pdr \nu}\, U_i \, dx + \int_\gw \left(K a(u + u_i)U_i^2  -\beta (u +  u_i)U_i V_i - K b\,(u^2 + u u_i + u_i^2) U_i^2 \right) dx \\
		+ &\, \left(K^2 + \frac{1}{2r} (q - K)^2\right) \|U_i (t)\|^2 + \frac{1}{4}\|V_i (t)\|^2 + \frac{r}{2}\|W_i (t)\|^2, \quad t > 0.
		\end{split}
\eeq
By the the boundary coupling condition \eqref{nbc}, the boundary integral in \eqref{eG} yields
\beq \bl{bdt}
	\begin{split}
	&\int_{\G} K\frac{\pdr U_i}{\pdr \nu}\, U_i \, dx = K\int_{\G}\, \sum_{i = 1}^m p [(u_i - u) - (u - u_i)]\, U_i \, dx \\[3pt]
	= &\, - 2K p \int_{\G_i} U_i^2(t, x)\, dx - 2Kp \int_{\G \backslash (\G_0 \cup \G_i)} u^2(t, x)\, dx
	\end{split}
\eeq	
for $1 \leq i \leq m$. We estimate another integral term on the right-hand side of \eqref{eG}, 
\beq \bl{nlt}
	\begin{split}
		&\int_\gw \left(K a(u + u_i)U_i^2  -\beta (u +  u_i)U_i V_i - K b\,(u^2 + u u_i + u_i^2) U_i^2 \right) dx \\[3pt]
		\leq &\, \int_\gw \left(K a(u + u_i)U_i^2  - \beta (u +  u_i)U_i V_i - \frac{K b}{2}(u^2 + u_i^2) U_i^2 \right) dx \\[3pt]
		\leq &\, \int_\gw \left(K a(u + u_i)U_i^2  + 2\beta^2 (u^2 +  u_i^2)U_i^2 + \frac{1}{4} V_i^2 - \frac{K b}{2}(u^2 + u_i^2) U_i^2 \right) dx.
	\end{split}
\eeq
Now we choose the constant multiplier $K$ to be
	\beq \bl{lbd}
	K = \frac{8 \beta^2}{b} > 0.
	\eeq
Then \eqref{nlt} is reduced to
	\beq \bl{me}
	\begin{split}
		&\int_\gw \left(K a\,(u + u_i)U_i^2  -\beta (u +  u_i)U_i V_i - K b\,(u^2 + u u_i + u_i^2) U_i^2 \right) dx \\[5pt]
		\leq &\, \int_\gw \left(K a(u + u_i)U_i^2 + \frac{1}{4} V_i^2 - \frac{K b}{4}(u^2 + u_i^2) U_i^2 \right) dx \\[3pt]
		= &\, \frac{1}{4} \|V_i(t)\|^2 +  \int_\gw \left( a(u + u_i) - \frac{b}{4}(u^2 + u_i^2) \right) KU_i^2\, dx \\
		= &\, \frac{1}{4} \|V_i(t)\|^2 +  \int_\gw \left[\frac{2a^2}{b} - \left(\frac{a}{b^{1/2}} - \frac{b^{1/2}}{2}\, u \right)^2 - \left(\frac{a}{b^{1/2}} - \frac{b^{1/2}}{2}\, u_i \right)^2 \right] KU_i^2\, dx \\
		\leq &\, \frac{1}{4} \|V_i(t)\|^2 + \frac{2K a^2}{b} \|U_i (t)\|^2 .
	\end{split}
	\eeq
Substitute \eqref{bdt} and \eqref{me} into \eqref{eG}. Then for $1 \leq i \leq m$ it holds that
\begin{equation} \bl{meq}
	\begin{split}
	&\frac{1}{2} \frac{d}{dt} (K \|U_i (t)\|^2 + \|V_i (t)\|^2 + \|W_i (t)\|^2) + 2Kp \int_{\G_i} U_i^2(t, x)\, dx \\
	&+ 2Kp \int_{\G \backslash  (\G_0 \cup \G_i)} u^2(t, x)\, dx + dK \, \|\nb U_i (t)\|^2 + \|V_i (t)\|^2 + r\, \|W_i (t)\|^2 \\
	\leq &\, \left(K^2 + \frac{K a^2}{b} + \frac{1}{2r} (q - K)^2\right) \|U_i (t)\|^2, \quad t > 0.
	\end{split}
\end{equation}

Step 2. By Poincar\'{e} inequality, there exist positive constants $\eta_1$ and $\eta_2$ depending only on the spatial domain $\gw$ and its dimension such that 
\beq \bl{Pcr}
	\eta_1 \|U_i (t)\|^2 \leq \|\nb U_i (t)\|^2 + \eta_2 \left(\int_\gw U_i (t, x)\, dx\right)^2, \quad 1 \leq i \leq m.
\eeq
On the other hand, Theorem \ref{Hab} with \eqref{M} and \eqref{lsp} confirm that 
\beq \bl{Lsup}
	\limsup_{t \to \infty}\, \left[\|g(t)\|^2 + \sum_{i=1}^m \|g_i(t)\|^2\right] \leq \frac{1 + m}{r^* \min \{C_1, 1\}}\left(C_2 + \frac{C_1^2}{32}\right) |\gw |.
\eeq	
Note that
$$
	 \|U_i (t)\|^2 \leq  2 (\|u(t)\|^2 + \|u_i(t)\|^2) \leq 2 \left(\|g(t)\|^2 + \sum_{i=1}^m \|g_i(t)\|^2\right).
$$
Then it follows from \eqref{Pcr} and \eqref{Lsup} that, for any given bounded set $B \subset H$ and any initial data $g^0, g_i^0 \in B$, we have
\beq \bl{Keq}
	\begin{split}
	& \frac{d}{dt} (K \|U_i (t)\|^2 + \|V_i (t)\|^2 + \|W_i (t)\|^2) + 4Kp \int_{\G_i} U_i^2(t, x)\, dx  \\[5pt]
	& + 2\,\eta_1 dK\, \|U_i (t)\|^2 + \|V_i (t)\|^2 + r \|W_i (t)\|^2 \\[3pt]
	\leq &\, 2\eta_2\, dK \left(\int_\gw U_i (t, x)\, dx\right)^2 + \left(K^2 + \frac{K a^2}{b} + \frac{1}{2r} (q - K)^2\right) \|U_i (t)\|^2  \\
	\leq &\, 2\eta_2\, dK |\gw | \|U_i (t)\|^2 + 2\left(K^2 + \frac{K a^2}{b} + \frac{1}{2r} (q - K)^2\right) \|U_i (t)\|^2. \\
	\leq &\, \frac{4(1 + m)}{r^* \min \{C_1, 1\}}\left(C_2 + \frac{C_1^2}{32}\right) |\gw | \left[\eta_2\, dK |\gw | + \left(K^2 + \frac{K a^2}{b} + \frac{1}{2r} (q - K)^2\right) \right]
	\end{split}
\eeq
for $t > T_B$. The differential inequality \eqref{Keq} is written as 
\beq  \bl{Synq}
	\begin{split}
	& \frac{d}{dt} (K \|U_i (t)\|^2 + \|V_i (t)\|^2 + \|W_i (t)\|^2) + 4Kp \int_{\G_i} U_i^2(t, x)\, dx  \\[3pt]
	& + 2\,\eta_1 dK\, \|U_i (t)\|^2 + \|V_i (t)\|^2 + r \|W_i (t)\|^2 < 4KR\,|\gw |, \quad t > T_B.
	\end{split}
\eeq
The constants $K = 8\beta^2/b$ in \eqref{lbd} and $R > 0$ in \eqref{R} are independent of initial data. 

	Under the condition that the stimulation signal strength of the boundary coupling $p\int_{\G_i} U_i^2(t, x)\, dx, 1 \leq i \leq m,$ satisfies \eqref{SC}, there exists a sufficiently large $\tau (g^0, g_i^0) > 0$ depending on the initial data such that the following two inequalities are satisfied:
	\beq \bl{QQ}
		\|g (\tau, g^0)\|^2 \leq Q, \quad \|g_i (\tau, g_i^0)\|^2 \leq Q, \;\; 1 \leq i \leq m,
	\eeq
where the constant $Q$ is in \eqref{abs}, and the threshold crossing inequality
	\beq \bl{pl}
	  4Kp \int_{\G_i} U_i^2(t, x)\, dx > 4KR\,|\gw |,  \quad t > \tau.
	\eeq
The inequality \eqref{pl} signifies that the boundary coupling effect $p \int_{\G_i} U_i^2(t, x)\, dx$ exceeds the synchronization threshold $R |\gw|$. Therefore, from \eqref{Synq} we have the differential inequalities: For $i = 1, \cdots, m$,
\beq \bl{Gwq}
	\begin{split}
	& \frac{d}{dt} (K \|U_i (t)\|^2 + \|V_i (t)\|^2 + \|W_i (t)\|^2) \\[5pt]
	&+ \min \{2\eta_1 d, 1, r\} (K\, \|U_i (t)\|^2 + \|V_i (t)\|^2 + \|W_i (t)\|^2) \\
	\leq & \frac{d}{dt} (K \|U_i (t)\|^2 + \|V_i (t)\|^2 + \|W_i (t)\|^2) \\[5pt]
	&+ 2\,\eta_1 dK\, \|U_i (\tau)\|^2 + \|V_i (\tau)\|^2 + r \|W_i (\tau)\|^2 < 0, \quad t > \tau.
	\end{split}
\eeq
Finally we apply the Gronwall inequality to \eqref{Gwq} and reach the conclusion that for all $i = 1, \cdots, m$,
	\beq \bl{dsyn}
	\begin{split}
		&K \|U_i (t)\|^2 + \|V_i (t)\|^2 + \|W_i (t)\|^2   \\[5pt]
		\leq &\, e^{- \mu (t - \tau)} (K \|U_i (\tau)\|^2 + \|V_i (\tau)\|^2 + \|W_i (\tau)\|^2) \\[5pt]
		\leq &\, 2e^{- \mu (t - \tau)} \max \{K, 1\} Q \to 0, \quad \text{as} \;\; t \to \infty,
	\end{split}
	\eeq
where $\mu = \min \{2\eta_1 d, 1, r\}$ is the uniform rate. Thus for any $j, k = 1, \cdots, m$, we have
\beq \bl{gjk}
	\begin{split}
	&\sup_{g_j^0, g_k^0 \in H} \left\{ \limsup_{t \to \infty} \|g_j (t, g_j^0) - g_k (t, g_k^0)\|_H \right\} \\
	\leq &\,\sup_{g_j^0, g^0 \in H} \left\{ \limsup_{t \to \infty} \|g_j (t, g_j^0) - g (t, g^0)\|_H \right\} \\
	+ &\,\sup_{g_k^0, g^0 \in H} \left\{ \limsup_{t \to \infty} \|g (t, g^0) - g_k (t, g_k^0)\|_H \right\} \to 0, \quad \text{as}\;\; t \to \infty.
	\end{split}
\eeq
Therefore it is proved that
	$$
	deg_s (\text{H})= \sum_{j = 0}^m \sum_{k=0}^m \, \sup_{g_j^0, g_k^0 \in L^2(\gw, \mR^3)} \, \left\{\limsup_{t \to \infty} \|g_j (t) -g_k(t)\|_{L^2(\gw, \mR^3)}\right\} = 0.
	$$
Here $g_0 (t, g_0^0) = g(t, g^0)$ for $i = 0$. It shows that the boundary coupled Hindmarsh-Rose neuron network generated by \eqref{pb} is asymptotically synchronized in the space $H = L^2 (\gw, \mR^3)$ at a uniform exponential rate. The proof is completed.
\end{proof}

Remark 1.
The presentation of this paper shows the asymptotic synchronization of a boundary coupled Hindmarsh-Rose neuron network locally of multiple neurons around a central neuron. This theory can be directly extended to a large-scale neuron network in the sense that each involved neuron is viewed as a central neuron with its own neighbor neurons. 

Remark 2.
Although there are mathematical models and many studies of neuron network in terms of ordinary differential equations, biologically the partly diffusive partial-ordinary differential equations will be more realistic for modeling the dynamics of neuron network because the neuron coupling and neuronal signal transmission usually take place on the boundary of the cell domain through bio-electrical potential stimulation signals which is related only to the first component $u$-equations. 

The main theorem in this paper provides a sufficient condition for realization of the asymptotic synchronization of this kind boundary coupled neuron network. The threshold for triggering the synchronization may possibly be reduced through further investigations.

\bibliographystyle{amsplain}

\begin{thebibliography}{99}
	
	\bibitem{BRS}
	R.J. Buters, J. Rinzel and J.C. Smith, {\em Models respiratory rhythm generation in the pre-B\"{o}tzinger complex, I. Bursting pacemaker neurons}, J. Neurophysiology, \textbf{81} (1999), 382--397.
	
	\bibitem{CK}
	T.R. Chay and J. Keizer,  {\em Minimal model for membrane oscillations in the pancreatic beta-cell}, Biophysiology Journal, \textbf{42} (1983), 181--189.
	
	\bibitem{CV}
	V. V. Chepyzhov and M. I. Vishik, {\em Attractors for Equations of Mathematical Physics},  AMS Colloquium Publications, Vol. \textbf{49}, AMS, Providence, RI, 2002.
	
	\bibitem{CS}
	L.N. Cornelisse, W.J. Scheenen, W.J. Koopman, E.W. Roubos and S.C. Gielen, {\em Minimal model for intracellular calcium oscillations and electrical bursting in melanotrope cells of Xenopus Laevis}, Neural Computations, \textbf{13} (2000), 113--137.
	
	\bibitem{DJ}
	M. Dhamala, V.K. Jirsa and M. Ding, {\em Transitions to synchrony in coupled bursting neurons}, Physical Review Letters, \textbf{92} (2004), 028101.
	
	\bibitem{ET}
	G.B. Ementrout and D.H. Terman, {\em Mathematical Foundations of Neurosciences}, Springer, 2010.
	
	\bibitem{FH}
	R. FitzHugh, {\em Impulses and physiological states in theoretical models of nerve membrane}, Biophysical Journal, \textbf{1} (1961), 445--466.
	
	\bibitem{HR}
	J.L. Hindmarsh and R.M. Rose, {\em A model of neuronal bursting using three coupled first-order differential equations}, Proceedings of the Royal Society London, Ser. B: Biological Sciences,  \textbf{221} (1984), 87--102.
	
	\bibitem{HH}
	A. Hodgkin and A. Huxley, {\em A quantitative description of membrane current and its application to conduction and excitation in nerve}, J. Physiology, Ser. B,  \textbf{117} (1952), 500--544.
	
	\bibitem{IG}
	G. Innocenti and R. Genesio, {\em On the dynamics of chaotic spiking-bursting transition in the Hindmarsh-Rose neuron}, Chaos, \textbf{19} (2009), 023124.
	
	\bibitem{EI}
	E.M. Izhikecich, {\em Dynamical Systems in Neuroscience: The Geometry of Excitability and Bursting}, MIT Press, Cambridge, Massachusetts, 2007.
	
	\bibitem{MFL}
	S.Q. Ma, Z. Feng and Q. Lu, {\em Dynamics and double Hopf bifurcations of the Rose-Hindmarsh model with time delay}, International Journal of Bifurcation and Chaos, \textbf{19} (2009), 3733--3751.
	
	\bibitem{Phan}
C. Phan, {\em Random attractor for stochastic Hindmarsh-Rose equations with multiplicative noise}, arXiv: 1908.01220, 2019, to appear in Discrete and Continuous Dynamical Systems.

	\bibitem{PY}
C. Phan and Y. You, {\em Exponential attractors for Hindmarsh-Rose equations in neurodynamics}, arXiv: 1908.05661, submitted for journal publication, 2019.
	
	\bibitem{PYS}
	C. Phan, Y. You and J. Su, {\em Global attractors for Hindmarsh-Rose equations in neurodynamics}, arXiv: 1907.13225, submitted for journal publication, 2019.
	
	\bibitem{PYY}
	C. Phan and Y. You, {\em Random attractor for stochastic Hindmarsh-Rose equations with additive noise}, to appear in Journal of Dynamics and Differential Equations.
		
	\bibitem{SY}
	G. R. Sell and Y. You, {\em Dynamics of Evolutionary Equations}, Applied Mathematical Sciences, Volume \textbf{143}, Springer, New York, 2002.
	
	\bibitem{SPH}
	J. Su, H. Perez-Gonzalez and M. He, {\em Regular bursting emerging from coupled chaotic neurons}, Discrete and Continuous Dynamical Systems, Supplement 2007, 946--955.
	
	\bibitem{Tm}
	R. Temam, {\em  Infinite Dimensional Dynamical Systems in Mechanics anf Physics}, 2nd edition, Springer, New York, (2013).
	
	\bibitem{Tr}
	D. Terman, {\em Chaotic spikes arising from a model of bursting in excitable membrane}, J. Appl. Math., \textbf{51} (1991), 1418--1450.
	
	\bibitem{WS}
	Z.L. Wang and X.R. Shi, {\em Chaotic bursting lag synchronization of Hindmarsh-Rose system via a single controller}, Applied Mathematics and Computation, \textbf{215} (2009), 1091--1097.
	
	\bibitem{Su}
	F. Zhang, A. Lubbe, Q. Lu and J. Su, {\em On bursting solutions near chaotic regimes in a neuron model}, Discrete and Continuous Dynamical Systems, Ser. S,  \textbf{7} (2014), 1363--1383.
\end{thebibliography}

\end{document}